\documentclass[a4paper,11pt]{article}

\usepackage{amsmath}
\usepackage{amsthm}

\usepackage[top=3.5cm, bottom=3.5cm, left=2.5cm, right=2.5cm]{geometry}

\usepackage[utf8]{inputenc}
\usepackage[T1]{fontenc}
\usepackage[french,english]{babel} 
\usepackage{enumerate}
\usepackage{enumitem}
\usepackage{color}
\usepackage[]{pdfpages}

\usepackage{titling}
\usepackage[colorlinks=false,linkcolor=blue]{hyperref}
\usepackage{csquotes}
\usepackage[nottoc, notlof, notlot, numbib]{tocbibind}
\usepackage[style=alphabetic,giveninits,doi=false,isbn=false,url=false,eprint=false,clearlang=true,sorting=nyt,maxbibnames=99]{biblatex}
\usepackage{mathtools}
\usepackage{thmtools}
\usepackage{bbm}
\usepackage{amssymb}
\usepackage{mathrsfs}
\usepackage{float}
\usepackage{stmaryrd}

\usepackage{setspace}
\usepackage{array}
\usepackage{tabularx}
\usepackage{ltablex} 
\usepackage{authblk}

\DeclareUnicodeCharacter{221E}{$\infty$}
\DeclareUnicodeCharacter{00B4}{\'}
\DeclareUnicodeCharacter{03BB}{$\eta_\sigma$}
\DeclareUnicodeCharacter{0393}{$\Gamma$}
\DeclareUnicodeCharacter{211D}{$\mathbb{R}$}
\DeclareUnicodeCharacter{2212}{-} 

\newcommand{\A}{\mathcal{A}}
\newcommand{\R}{\mathbb{R}}

\newcommand{\E}{\mathbb{E}}

\renewcommand{\d}{\mathrm{d}}
\renewcommand{\P}{\mathbb{P}}

\DeclareMathOperator{\argmin}{argmin}

\newtheorem{theorem}{Theorem}[section]
\newtheorem{proposition}[theorem]{Proposition}
\newtheorem{lemma}[theorem]{Lemma}

\theoremstyle{definition}
\newtheorem{definition}[theorem]{Definition}

\theoremstyle{remark}
\newtheorem{rem}[theorem]{Remark}

\theoremstyle{plain}
\newtheorem{assumption}{Assumption}

\title{\bf Existence and global Lipschitz estimates
for unbounded classical solutions of a Hamilton-Jacobi equation}
\date{}

\author{Louis-Pierre \textsc{Chaintron}}
\affil{
\small{DMA, École normale supérieure, Université PSL, CNRS, 75005 Paris, France.
CERMICS, École des ponts, 77420 Champs-sur-Marne, France.
Inria, Team M$\sf{\Xi}$DISIM, Inria Saclay, 91128 Palaiseau, France. 
} 
\url{louis-pierre.chaintron@ens.fr}
}

\begin{document} 
\maketitle

\begin{abstract}
The purpose of this article is to prove existence, uniqueness and uniform gradient estimates for unbounded classical solutions of a Hamilton-Jacobi-Bellman equation. 
Such an equation naturally arises in stochastic control problems.
Contrary to the classical literature which handles the case of bounded regular coefficients, we only impose Lipschitz regularity conditions, allowing for a linear growth of coefficients.
This latter regularity assumption is natural in a probabilistic setting. 
In principle, this assumption is compatible with global Lipschitz regularity for the solution. 
However, to the best of our knowledge, this useful result had not been established before.
The proof that we provide relies on classical methods from the viscosity solution theory, combining the Ishii-Lions method \cite{ishii1990viscosity} for uniformly elliptic equations with ideas from the weak Bernstein method \cite{barles1991weak}.
\end{abstract}

\tableofcontents

\section{Introduction} \label{sec:intro}

For any integer $d \geq 1$ and $T >0$, we consider the following Hamilton-Jacobi-Bellman equation:
\begin{equation} \label{eq:HJB}
\begin{cases}  
\partial_t u(t,x) + c u(t,x) + \tfrac{1}{2} \mathrm{Tr} [ \sigma \sigma (t,x) D^2 u (t,x) ] + \inf_{\alpha \in \mathcal{A}} [ b (t,x,\alpha) \cdot D u (t,x) + f(t,x,\alpha) ] = 0, \\
u(T,x) = g(x), \quad (t,x) \in (0,T) \times \R^d.
\end{cases}
\end{equation} 
for $\sigma : [0,T] \times \R^d \rightarrow \R^{d \times d}$, $b : [0,T] \times \R^d \times \A \rightarrow \R^d$, $c : [0,T] \times \R^d \times A \rightarrow \R$, $f : [0,T] \times \R^d \times \A \rightarrow \R$ and $c$ a constant. 
The control set $\A$ is a possibly unbounded closed domain in $\R^d$.
The semi-linear equation \eqref{eq:HJB} naturally arises in finite-time horizon stochastic control problems.
More  specifically, on a filtered probability space $(\Omega,(\mathcal{F}_t)_{0\leq t \leq T},\P)$, we consider the controlled dynamics in $\R^d$,
\begin{equation} \label{eq:control}
\begin{cases}
\d X^{t,x,\alpha}_s = b(t,X^{t,x,\alpha}_s,\alpha_s) \d s + \sigma(s,X^{t,x,\alpha}_s) \d B_s, \quad t \leq s \leq T, \\ 
X^{t,x,\alpha}_t = x,
\end{cases}
\end{equation}
where $(B_s)_{0 \leq s \leq T}$ is a $(\mathcal{F}_t)_{0\leq t \leq T}$-Brownian motion, and the control process $\alpha := (\alpha_s)_{0 \leq s \leq T}$ is any $(\mathcal{F}_t)_{0\leq t \leq T}$-adapted $\A$-valued stochastic process.
The related cost function is
\begin{equation} \label{eq:cost}
V(t,x) := \inf_\alpha \E \int_t^T e^{c(s-t)} f(s,X^{t,x,\alpha}_s,\alpha_s) \d s + g(X^{t,x,\alpha}_T). 
\end{equation} 
If \eqref{eq:HJB} has a $C^{1,2}$ solution $u$ which satisfies suitable bounds, then it is classical that $u = V$ (see e.g. \cite[Theorem III.8.1]{fleming2006controlled}), and it is possible to compute optimal controls for \eqref{eq:control}-\eqref{eq:cost} by looking for feed-back controls $\overline{\alpha}_s := \overline{\alpha}(s,X^{t,x,\overline\alpha}_s)$ such that
\begin{equation} \label{eq:PMP}
\overline{\alpha}(t,x) \in \argmin_{\alpha \in \A} [ b (t,x,\alpha) \cdot D u (t,x) + f(t,x,\alpha) ]. 
\end{equation} 
However, having such a classical solution often requires strong assumptions on the coefficients.
A standard reference on this issue is \cite[Theorem 6.2]{fleming1975applications}, which provides classical solutions for \eqref{eq:HJB} when $b$, $f$ and $\sigma$ are at least $C^{1,2}_b$ (bounded functions with bounded continuous derivatives) and the control set $\A$ is compact. 
In the \cite{fleming1975applications} setting, the solution of \eqref{eq:HJB} is a $C^{1,2}_b$ function.
Our main objective is the following result:

\begin{theorem} \label{thm:quad}
If $b_2$, $\sigma$ and $f_2$ are Lipschitz-continuous functions, with $\sigma$ bounded and $\sigma \sigma^\top$ uniformly elliptic, then   
\begin{equation} \label{eq:quadHJB}
\partial_t u (t,x) +  \tfrac{1}{2} \mathrm{Tr} [ \sigma \sigma^\top (t,x) D^2 u (t,x) ] + b_2 (t,x) \cdot D u - \tfrac{1}{2} \lvert \sigma^\top (t,x) D u (t,x) \rvert^2 + f_2 (t,x) = 0,
\end{equation}
has a unique linear growth $C^{1,2}$ solution $u$, which has a uniformly bounded gradient.
\end{theorem}

These Lipschitz assumptions on the coefficients are natural in a probabilistic setting.
Indeed, from the It\^{o} theorem \cite{stroock1997multidimensional}, a natural sufficient assumption for \eqref{eq:control} to be well-posed is that both $b$ and $\sigma$ are globally Lipschitz ($\alpha$ must satisfy some moment bound too).   
In particular, this framework allows for a linear growth of the coefficients, which implies that the solution of \eqref{eq:quadHJB} can have linear growth.
Equation \eqref{eq:quadHJB} is a particular case of \eqref{eq:HJB} when specifying $b(t,x,\alpha) = \sigma^\top(t,x) \alpha + b_2(t,x) $, $f(t,x,\alpha) = \tfrac{1}{2} \lvert \alpha \rvert^2 + f_2(t,x)$ together with the unbounded control set $\A = \R^d$. 
In this case, \eqref{eq:PMP} reduces to $\overline{\alpha}(t,x) = -\sigma^\top(t,x) D u(t,x)$. 
In particular, using the global bound on $\sigma$, Theorem \ref{thm:quad} proves that optimal controls are globally bounded. 
Theorem \ref{thm:quad} will be a consequence of Theorem \ref{thm:Main} below, which is concerned with more general equations of type \eqref{eq:HJB}. 
We now summarise the results in the literature that are close to ours. 

\medskip
Well-posedness for \eqref{eq:quadHJB} is related to well-posedness for the linear parabolic equation
\begin{equation} \label{eq:Hopf-Cole}
\partial_t v + b_2 \cdot D v + \tfrac{1}{2} \mathrm{Tr} [ \sigma \sigma^{\top} D^2 v ] - f_2 v = 0,
\end{equation}
which can be obtained from \eqref{eq:quadHJB} using the Cole-Hopf transform $v = e^{-u}$.
Classical references for \eqref{eq:Hopf-Cole} are \cite{gilbarg1977elliptic,ladyzhenskaia1988linear,friedman2008partial} which provide classical well-posedness and Hölder-regularity estimates  when the coefficients $b$, $\sigma$ and $f_2$ are regular enough, either in bounded domains, or in $\R^d$ when the coefficients $b$, $\sigma$ and $f_2$ are also globally bounded.
Their work also handles Lipschitz non-linear perturbations of \eqref{eq:Hopf-Cole}. 
This latter extension includes \eqref{eq:HJB} when the control set $\A$ is compact. 
Similar results for \eqref{eq:HJB} are proved in \cite{krylov1987nonlinear,krylov2008controlled} when $f$ is bounded uniformly in $\alpha$ and the coefficients are twice differentiable.

Well-posedness for linear parabolic equations with unbounded coefficients is a long-standing topic.
Historical approaches \cite{bodanko1966probleme,johnson1971priori,besala1975existence} rely on the parametrix method and provide existence and bounds on a Green function for \eqref{eq:Hopf-Cole} under a Lyapunov-type assumption, but they also require heavy differentiability conditions on the coefficients.
See also \cite{deck2002parabolic} for a more recent generalisation.
On the other side, the solution of \eqref{eq:Hopf-Cole} admits a stochastic representation using the Feynman-Kac formula.
This latter formula can be used in turn to build a classical solution for \eqref{eq:Hopf-Cole}, but this approach \cite{friedman2012stochastic} also requires heavy differentiability assumptions on the coefficients.
Another noticeable limitation is the assumption that $f_2$ is bounded from below, which is necessary to guarantee that solutions do not explode.
The removal of this restriction on $f_2$ can imply the loss of the maximum principle and non-uniqueness situations. 

More recently, the works
\cite{lunardi1998schauder,fornaro2004gradient,bertoldi2004gradient,bertoldi2007gradient,hieber2007second}... provide classical well-posedness, existence of Green functions and regularity estimates for linear parabolic equations with unbounded coefficients.
These results rely on semi-group methods, under a fairly general Lyapunov-type condition.
This framework also extends to Lipschitz non-linear perturbations, see e.g. \cite{addona2017hypercontractivity}.
However, these results mostly focus on globally bounded solutions, a situation which cannot be expected in our framework for \eqref{eq:HJB}.
In \cite{ito2001existence}, existence of generalised strong solutions with quadratic growth is shown for a class of Hamilton-Jacobi equations whose non-linearity in $D u$ includes the one in \eqref{eq:quadHJB}.
\cite[Theorem 2.3]{ito2001existence} provides a uniform bound on $D u$ in some situations where $f_2$ is Lipschitz and $u$ is globally bounded, and \cite[Theorem 2.7]{ito2001existence} proves linear growth for $D u$ when $f_2$ has quadratic growth. 
However, these results require a strong confinement assumption on the drift $b_2$ and do not provide classical solutions.
Concerning gradient estimates for \eqref{eq:Hopf-Cole} with unbounded coefficients and uniform in time estimates, we also mention \cite{priola2006gradient,conforti2022coupling} whose results rely on stochastic methods based on reflection coupling.

To the best of our knowledge, the closest results to our setting are the ones in \cite{rubio2011existence}, 
which provide classical well-posedness for polynomial-growth solutions of \eqref{eq:Hopf-Cole} under polynomial-growth conditions on the coefficients.
The proof strategy combines both stochastic representation formulae and classical domain truncation methods.
Moreover, this work provides polynomial-growth solutions for \eqref{eq:HJB} when the control set $\A$ is compact. 
At this point, a possible approach to well-posedness for \eqref{eq:HJB} is to truncate the non-linearity to make it enter the framework of \cite{rubio2011existence}, by restricting the unbounded control set $\A$ to a bounded subset. 
If the solution of the truncated equation satisfies a Lipschitz estimate independent of the truncation, then \eqref{eq:HJB} and its truncated version will coincide for a sufficiently high truncation threshold.
This amounts to proving a priori global Lipschitz estimates for $C^{1,2}$ solutions of \eqref{eq:HJB}.

An efficient method to prove global Lipschitz estimates on the solution $u$ of an elliptic or parabolic equation of the kind
\begin{equation} \label{eq:GeneralHJ}
\partial_t u(t,x) + H(t,x,D u(t,x),D^2 u(t,x)) = 0,
\end{equation} 
is the classical Bernstein method (see e.g. \cite{gilbarg1977elliptic}).
By differentiating the equation, this approach looks for a linear equation satisfied by $w := \lvert D u \rvert^2$, and estimates are obtained using standard maximum principle-type arguments.
In particular, this method requires extra-regularity properties on $u$ for the differential of \eqref{eq:GeneralHJ} to make sense, or it needs to approximate \eqref{eq:GeneralHJ} by a smoothed version and to use continuation methods.
A successful application of this approach can be found in \cite{daudin2022optimal}, under assumptions on $H$ which guarantee that $u$ is $C^{1,2}_b$. 
This method is extended to linear-growth solutions of \eqref{eq:quadHJB} when $\sigma = \mathrm{Id}$ in \cite{daudin2023optimal}, well-posedness resulting from a fixed-point method.
However, \cite{daudin2022optimal,daudin2023optimal} cannot handle more general $\sigma$, because the Bernstein method they use requires the structural assumption that
\begin{equation} \label{eq:Structural}
\forall p \in \R^d, \quad \sup_{(t,x) \in [0,T] \times \R^d} \lvert D_x H(t,x,p,0) \rvert \leq C [ 1 + \lvert p \rvert ], 
\end{equation} 
for some constant $C > 0$. 
This assumption is not satisfied by \eqref{eq:quadHJB} in general if $\sigma \neq \mathrm{Id}$. 

Another well-posedness approach for \eqref{eq:HJB} is the viscosity solution theory \cite{crandall1983viscosity,crandall1992user,barles1994solutions}.
Under suitable structural assumptions like \eqref{eq:Structural}, existence and comparison principle (hence uniqueness) for bounded uniformly continuous solutions of \eqref{eq:HJB} are by now standard results (see e.g. \cite[Chapter 6]{fleming2006controlled}).
Extensions to linear-growth (or more general growth) situations can be found in e.g. \cite{ishii1984uniqueness,ishii1989uniqueness,crandall1989existence,crandall1990quadratic,barles2003uniqueness}.
A linear growth situation similar to ours is \cite{giga1991comparison}, but still keeping a structural condition of the kind \eqref{eq:Structural}.
A general result that covers existence and uniqueness for viscosity solutions of \eqref{eq:HJB} in our Lipschitz setting (and far more general settings) is \cite{da2006uniqueness}.

The next step is to prove Lipschitz estimates on the obtained viscosity solution. 
A seminal work to prove regularity estimates is \cite{ishii1990viscosity}: under uniformly elliptic or parabolic assumptions, the Ishii-Lions method provides Hölder or Lipschitz estimates, see \cite{barles2008c0} for a summarised presentation.
When Lipschitz estimates are available, the Ishii-Lions method even implies semi-concavity estimates. 
For Lipschitz estimates, the main idea is to prove that
\begin{equation*} 
\sup_{(t,x,y) \in [0,T] \times \R^d \times \R^d} u(t,x) - u(t,y) - K \lvert x - y \rvert \leq 0 \end{equation*}
for large enough $K$, by reasoning by contradiction and combining the PDE and the second-order optimality conditions at a maximum point. 
Similarly, the analogous of the Bernstein method can be developed for viscosity solutions: this is the weak Bernstein method \cite{barles1991weak}.
Contrary to the classical Bernstein method, this method does not require any additional differentiability assumption on the solution, but it requires similar structure conditions on the equation.
The analogous using viscosity solutions of gradient estimates in \cite{priola2006gradient}
is done in \cite{porretta2013global}.
Lipschitz estimates are obtained for globally bounded solutions of an equation similar to \eqref{eq:HJB} in \cite{papi2002generalized,papi2003regularity}. 
In \cite{giga1991comparison}, Lipschitz estimates and concavity-preserving properties are shown for linear growth viscosity solutions of \eqref{eq:GeneralHJ} in the case where $H$ does not depend on $(t,x)$.

We are here concerned with classical solutions of \eqref{eq:HJB} rather than viscosity ones.
After proving existence, uniqueness, and Lipschitz estimates for viscosity solutions, it is possible to obtain semi-concavity and higher-order regularity results using \cite{caffarelli1989interior,wang1989regularity,imbert2006convexity}... 
This strategy is applied to \eqref{eq:GeneralHJ} in \cite{daudin2022optimal}, under \eqref{eq:Structural} and assumptions that guarantee the global boundedness of solutions.  
In our situation, the aforementioned strategy based on the results of \cite{rubio2011existence} in the case of a compact control set $\A$ avoids to deal with viscosity solutions. 
Nevertheless, a priori global Lipschitz estimates in our work are proven using the Ishii-Lions (since classical solutions are special cases of viscosity solutions).

\section{Assumptions and main result} \label{sec:results}

In the following, the control set $\A$ is a possibly unbounded closed domain in $\R^d$.
The notation $\lvert x \rvert := \sqrt{x^\top x}$ refers to the usual Euclidean norm on $\R^d$.
Similarly, we will use the Euclidean norm $\lvert \sigma \rvert := \sqrt{ \mathrm{Tr}[\sigma \sigma^\top] }$ on matrices.
We write the running cost $f$ as
\[ f(t,x,\alpha) = f_1(t,x,\alpha) + f_2(t,x). \]
In the following, we restrict ourselves to
\begin{equation} \label{eq:decomp}
b(t,x,\alpha) = b_1(t,x,\alpha) + b_2(t,x), 
\end{equation} 
the $b_1$ part being uniformly bounded in $x$.
Although we can handle some more general equations, we emphasise that the main purpose of the following results is to prove Theorem \ref{thm:quad} in the Introduction.
Let us now detail the needed regularity assumptions for our more general results. 

\begin{assumption}[Regularity and growth] \label{ass:Lip}
The coefficients $b_1$, $c$, $f_1$ (resp. $b_2$, $\sigma$ and $f_2$) are continuous functions on $[0,T] \times \R^d \times \A$ (resp. $[0,T] \times \R^d$), which satisfy the following conditions.
\begin{itemize}
    \item Regularity and growth for $b_1$ and $b_2$: there exists $L_b>0$ such that for every $(t,x,y,\alpha)$ in $[0,T] \times \R^d \times \R^d \times \A$,  \begin{equation}\label{eq:bipolaire}
    \lvert b_1(t,x,\alpha) - b_1(t,y,\alpha) \rvert \leq L_b ( 1 + \lvert \alpha \rvert ) \lvert x - y \rvert, \quad \lvert b_2(t,x) - b_2(t,y) \rvert \leq L_b \lvert x - y \rvert,
    \end{equation}
    \[ \lvert b_1 (t,x,\alpha) \rvert \leq L_b [1 + \lvert \alpha \rvert], \quad \lvert b_2(t,x) \rvert \leq L_b [1 + \lvert x \rvert ]. \]
    \item Regularity and coercivity for $f_1$: there exist $L_{f_1}, c_{f_1}, c'_{f_1}, C_{f_1}, C'_{f_2} > 0$ such that for every $(t,x,y,\alpha)$ in $[0,T] \times \R^d \times \R^d \times \A$,
    \begin{equation} \label{eq:Lipf1}
    \lvert f_1(t,x,\alpha) - f_1(t,y,\alpha) \rvert \leq L_{f_1} \lvert x - y \rvert, 
    \end{equation} 
    together with
    \begin{equation} \label{eq:Coercif}
    c_{f_1} \lvert \alpha \rvert^2 - c'_{f_1} \leq f(t,x,\alpha) \leq  C_{f_1} \lvert \alpha \rvert^2 + C'_{f_1}.
    \end{equation}
    \item Regularity and growth for $f_2$: there exists $L_{f_2} \in C([0,T],\R_+)$ such that for every $(t,x,y)$ in $[0,T] \times \R^d \times \R^d$, 
\begin{equation}\label{eq:Lipf2}
    \lvert f_2 (t,x) - f_2 (t,y) \rvert \leq L_{f_2}(t) \lvert x - y \rvert, \quad \lvert f_2 (t,x) \rvert \leq L_{f_2}(t) [ 1 + \lvert x \rvert ].
    \end{equation}
    \item Regularity and growth for $\sigma$ and $g$: there exist $L_\sigma, \, L_{g}>0$ such that for every $(t,x,y)$ in $[0,T] \times \R^d \times \R^d$, 
    \[ \lvert \sigma(t,x) - \sigma(t,y) \rvert \leq L_\sigma \lvert x - y \rvert^{1/2}, \quad \lvert g(x) - g(y) \rvert \leq L_g \lvert x -y \rvert. \]
\end{itemize}
\end{assumption}

As explained in Section \ref{sec:intro}, Lipschitz assumptions on the coefficients are natural in a probabilistic setting to guarantee well-posedness for \eqref{eq:control}. 
Assumption \ref{ass:Lip} is needed to prove Lipschitz estimates on the solution $u$ of \eqref{eq:HJB}.

\begin{assumption}[Further assumptions on $\sigma \sigma^\top$] \label{ass:Boundsigma}
$\phantom{A}$
\begin{itemize}
\item The matrix $\sigma \sigma^\top$ is uniformly elliptic: there exists $\eta_\sigma > 0$ such that for every $(t,x) \in [0,T] \times \R^d$, 
    \[ \forall \xi \in \R^d, \quad \xi^\top \sigma \sigma^\top(t,x) \, \xi \geq  \eta_\sigma \lvert \xi \rvert^2.  \] 
    \item The matrix $\sigma \sigma^\top$ is uniformly continuous in $x$: there exists a uniform continuity modulus $m_\sigma : \R_+ \rightarrow \R_+$ such that for every $(t,x,y)$ in $\R_+ \times \R^d \times \R^d$,
    \[ \big\lvert \sigma \sigma^\top (t,x) - \sigma \sigma^\top (t,y) \big\rvert \leq m_\sigma ( \lvert x - y \rvert ).  \]
    \end{itemize}
\end{assumption}

If \ref{ass:Lip} is satisfied, uniform continuity for $\sigma \sigma^\top$ always holds if $\sigma$ is globally bounded (as it is the case in Theorem \ref{thm:quad}).
An unbounded $\sigma$ which satisfies \ref{ass:Lip}-\ref{ass:Boundsigma} is e.g. $\sigma(t,x) = \sqrt{1+\lvert x \rvert} \mathrm{Id}$.
Assumption \ref{ass:Boundsigma} is needed to apply the Ishii-Lions method, and we did not manage to alleviate it. 
In the presence of degeneracy, existence of classical solutions to \eqref{eq:HJB} can be lost.

\begin{definition}[Local Hölder spaces]
For integers $n,n' \geq 1$ and $\beta,\beta' \in(0,1]$,
a function $\psi$ in $C^{n,n'}((0,T) \times \R^d)$ belongs to the local Hölder space $C^{n,n',\beta,\beta'}_\mathrm{loc}$ if for every compact set $F \subset (0,T) \times \R^d$, there exists $K_F > 0$ such that $\psi$ and all its derivatives satisfy
\begin{equation} \label{eq:Holder}
\forall (t,x),(s,y) \in F, \quad \lvert \psi(t,x) - \psi(s,y) \rvert \leq K_F [ \lvert t - s\rvert^\beta + \lvert x - y\rvert^{\beta'} ]. 
\end{equation}
\end{definition}

For every differentiable $\psi : [0,T] \times \R^d \rightarrow \R$, we define the set
\begin{equation} \label{eq:Apsi}
A [\psi] (t,x) := \argmin_{\alpha \in \A} [ b_1 (t,x,\alpha) \cdot D \psi(t,x) + f_1(t,x,\alpha) ]. 
\end{equation}
The following assumption is not required to prove Lipschitz estimates.
This assumption is copied from \cite[Assumption H3-(5)]{rubio2011existence} to obtain existence of classical solutions with Lipschitz non-linearity.

\begin{assumption} [Hölder regularity] \label{ass:Holder}
There exists $\beta \in (0,1)$ such that the coefficients $b$, $\sigma$ and $f$ belong to $C^{0,0,\beta,1}_\mathrm{loc}$.
The coefficients $b_1$ and $f_1$
are locally Lipschitz in $\alpha$ in the sense of \eqref{eq:Holder}.
Moreover, for every $\beta'\ \in \{\beta,1\}$ and every $\psi$ in $C^{0,1,\beta,\beta'}_\mathrm{loc}$, $A[\psi]$ is a singleton, defining a function of $(t,x)$ which belongs to $C^{0,0,\beta,\beta'}_\mathrm{loc}$. \\
Denoting by $B(0,R)$ the centred ball $B$ of $\R^d$ with radius $R$, we eventually require that this latter property on $A[\psi]$ holds when replacing $\A$ by $\A \cap B(0,R)$ for every $R \geq R_\A$, for a large enough $R_\A > 0$. 
\end{assumption}

In practice, regularity properties for $A[\psi]$ are easier to study when assuming convexity properties in $\alpha$ for the coefficients.
Once again, we point out that \ref{ass:Lip}-\ref{ass:Boundsigma}-\ref{ass:Holder} are satisfied in the setting of Theorem \ref{thm:quad}.

\begin{theorem}[Well-posedness and uniform gradient bound] \label{thm:Main}
Under \ref{ass:Lip}-\ref{ass:Boundsigma}-\ref{ass:Holder}, Equation \eqref{eq:HJB} has a unique linear growth solution $u$ in $C([0,T] \times \R^d) \cap C^{1,2,\beta,\beta}_\mathrm{loc}$.
Moreover, there exists a constant $K > 0$ such that
\begin{equation} \label{eq:HJBLip}
\sup_{(t,x) \in (0,T) \times \R^d} \lvert D u(t,x) \rvert \leq K.
\end{equation} 
The constant $K$ only depends on the regularity constants introduced in \ref{ass:Lip}-\ref{ass:Boundsigma}-\ref{ass:Holder}, and $K$ only depends on $f_2$ through $\lVert L_{f_2} \rVert_{L^1(0,T)}$.
\end{theorem}

The dependence on $f_2$ is emphasised to compare Theorem \ref{thm:Main} to
\cite[Theorem 1.1 and Lemma A.2]{daudin2023optimal}, where a similar dependence is obtained under the additional assumption \eqref{eq:Structural}.
In \cite[Lemma A.2]{daudin2023optimal}, this dependence is obtained using the classical Bernstein method.
However, this method does not work without \eqref{eq:Structural} and Theorem \ref{thm:Main} extends the results of   \cite{daudin2023optimal} to our setting.
In particular, following \cite{daudin2023optimal}, it is possible to extend Theorem \ref{thm:Main} to measure-valued source terms $f_2$ that are useful when dealing with constrained stochastic control problems.

Under \ref{ass:Lip}, for every $(t,x,\alpha)$ in $[0,T] \times \R^d \times \A$, the quantity to minimise in the r.h.s. of \eqref{eq:Apsi} is bounded from below by
\[ -L_b [ 1 + \lvert \alpha \rvert ] \lvert D \psi(t,x) \rvert + c_{f_1} \lvert \alpha \rvert^2 - c'_{f_1}, \]
and bounded from above by
\[ L_b [ 1 + \lvert \alpha \rvert ] \lvert D \psi(t,x) \rvert + C_{f_1} \lvert \alpha \rvert^2 + C'_{f_1}. \]
Both theses quantities only depend on $\lvert \alpha \rvert$.
As a consequence, there exists $L_\A > 0$ independent of $\psi$, such that for every $(t,x)$ in $[0,T] \times \R^d,$
\begin{equation} \label{eq:defLA}
\forall \overline{\alpha} \in A [\psi](t,x), \quad \lvert \overline{\alpha} \rvert \leq L_\A [ 1 + \lvert D \psi(t,x) \rvert ].  
\end{equation}
This property will be crucial in the proofs of Lipschitz estimates.
If $D u$ is bounded, then \eqref{eq:defLA} guarantees using \eqref{eq:PMP} that optimal controls for \eqref{eq:control}-\eqref{eq:cost} are globally bounded. 
In the following, all the estimates will only depend on $\A$ through $L_\A$.

\begin{rem}[About the coercivity assumption] 
The property \eqref{eq:defLA} is the only reason for the decomposition \eqref{eq:decomp} and \eqref{eq:Coercif}.
The lower bound in \eqref{eq:Coercif} is a coercivity assumption which is standard in stochastic control.
The dependence on $\lvert \alpha \rvert$ for the upper bound is not really restrictive, since this term is only meant to bound the quantity to minimise in \eqref{eq:Apsi} from above, independently of $x$. 
Any other assumption that guarantees \eqref{eq:defLA} could replace \eqref{eq:decomp}-\eqref{eq:Coercif}.
\end{rem}

\begin{rem}[Adding a linear term]
Dealing with non-constant $c(t,x,\alpha)$ is uneasy, because we need $c u$ to be globally-Lipschitz when $u$ is only globally Lipschitz.
This is very restrictive, but it would be possible to adapt the result for e.g. $c(t,x,\alpha) = \tfrac{c_0\lvert \alpha \rvert}{(1 + \lvert \alpha \rvert)(1+\lvert x \rvert)}$.
\end{rem}

\begin{rem}[Possible improvements of Lipschitz hypotheses]
The Lipschitz assumptions on the coefficients are two-fold. 
For instance, \eqref{eq:Lipf1} provides a control of $\lvert f_1(t,x,\alpha) - f_1(t,y,\alpha) \rvert$ when $\lvert x - y \rvert$ is large, together with local regularity when $\lvert x - y \rvert$ is small.
Both these properties are needed to obtain the analogous ones for $u$, which are gathered in the Lipschitz estimate \eqref{eq:HJBLip}: they respectively correspond to Lemma \ref{lem:detoriated} and Proposition \ref{pro:Lip} below.
The large $\lvert x - y \rvert$ part could be obtained by only imposing that
\[  \lvert f_1(t,x,\alpha) - f_1(t,y,\alpha) \rvert \leq 1 + L_{f_1} \lvert x - y \rvert + h(\lvert \alpha \rvert), \]
for any function $h \geq 0$. 
The small $\lvert x - y \rvert$ control could be softened into
\[  \lvert f_1(t,x,\alpha) - f_1(t,y,\alpha) \rvert \leq L_{f_1} ( 1 +  \lvert \alpha \rvert^2) \lvert x - y \rvert^\mu, \]
for any $\mu \in (0,1]$. 
The proof of Proposition \ref{pro:EstimLip} should then be modified by considering $\rho ( z ) = z^{1+\nu}$ with $\nu < \mu$, instead of $\rho ( z ) = z^{3/2}$.
We could split the condition \eqref{eq:bipolaire} similarly. 
Similarly, only the large $\lvert x - y \rvert$ control is needed for $\sigma$ in \ref{ass:Lip}, because the only needed regularity is given by the uniform continuity for $\sigma \sigma^\top$ in \ref{ass:Boundsigma}. 
We refrained ourselves from adding these subtleties in \ref{ass:Lip} to avoid writing tedious assumptions.
\end{rem}

\section{Proof of global Lipschitz estimates} \label{sec:proofs}

As explained in Section \ref{sec:intro}, we are going to prove a priori global Lipschitz estimates on linear growth $C^{1,2}$ solutions of \eqref{eq:HJB}.
We first start with a short lemma about the growth of such solutions.

\begin{lemma} \label{lem:Growth}
Let $u$ in $C([0,T] \times \R^d) \cap C^{1,2}((0,T) \times \R^d)$ be a quadratic growth solution of \eqref{eq:HJB}, in the sense that
\[ \exists C_u > 0: \forall (t,x) \in [0,T] \times \R^d, \quad \lvert u(t,x) \rvert \leq C_u [1 +\lvert x\rvert^2]. \]
Then under \ref{ass:Lip}, $u$ has linear growth:
\[ \forall (t,x) \in [0,T] \times \R^d, \quad \lvert u(t,x) \rvert \leq L [1 +\lvert x\rvert], \]
for a positive constant $L$ which does not depend on $u$ and only depends on $(\A, f_2)$ through $(L_\A, \lVert L_{f_2} \rVert_{L^1} )$.
\end{lemma}

\begin{proof}
The function $z \in \R^d \mapsto ( 1 + \lvert z \rvert^2 )^{1/2}$ is $C^2$ and convex, and thus has non-negative Hessian.
From \ref{ass:Lip}, $(t,x) \mapsto \sqrt{2} ( 1 + \lvert x \rvert^2 )^{1/2} [ \int_0^t L_{f_2} (s) \d s + L e^{t} ]$ is a super-solution of \eqref{eq:HJB} for large enough $L > 0$. 
Similarly, $(t,x) \mapsto \sqrt{2} ( 1 + \lvert x \rvert^2 )^{1/2} [ e^{-\delta t} - \int_0^t L_{f_2} (s) \d s ]$ is a sub-solution of \eqref{eq:HJB} for large enough $\delta > 0$.
The result then follows from the comparison principle proved in \cite[Theorem 2.1]{da2006uniqueness}.
\end{proof} 

Noticeably, the above result allows us to include quadratic growth solutions in the uniqueness result of Theorem \ref{thm:Main}. 
We now prove Lipschitz estimates by showing that
\begin{equation} \label{eq:LipWeak}
\sup_{t,x,y} u(t,x) - u(t,y) - K \lvert x - y \rvert \leq 0 \end{equation}
for large enough $K$.
The supremum is taken over $t \in [0,T]$ and $x,y \in \R^d$.
We reason by contradiction and we aim at combining the PDE and the second-order optimality conditions. 
The Ishii-Lions method provides a way to obtain a contradiction by making the second-order terms explode at a maximum point $(\overline{x},\overline{y})$.
A key feature for this approach is to prove that $\lvert \overline{x} - \overline{y} \rvert \rightarrow 0$ as $K \rightarrow +\infty$. 
Technical difficulties arise at this stage: there is no guarantee that the supremum in \eqref{eq:LipWeak} is finite and this supremum does not need to be realised at some point.
This latter difficulty is classically circumvented by adding regularising terms to the supremum and progressively removing them.
However, we need to know a priori that the supremum is finite to ensure that $\lvert \overline{x} - \overline{y} \rvert \rightarrow 0$ as $K \rightarrow +\infty$.
To do so, we reason as in \cite[Theorem 5.1]{crandall1992user} and we use the linear growth to first prove a deteriorated estimate.

\begin{lemma}[Deteriorated estimate] \label{lem:detoriated}
Let $u$ in $C([0,T] \times \R^d) \cap C^{1,2}((0,T) \times \R^d)$ be a solution of \eqref{eq:HJB} and $L_u > 0$ such that for every $(t,x) \in [0,T] \times \R^d$,
\begin{equation} \label{eq:DetLin}
\lvert u(t,x) \rvert \leq L_u [1 +\lvert x\rvert]. 
\end{equation} 
Under \ref{ass:Lip}, there exist constants $\tilde{K}, \tilde{M} > 0$ such that
\begin{equation} \label{eq:Deteriorated}
\sup_{(t,x,y) \in [0,T] \times \R^d \times \R^d} u(t,x) - u(t,y) - \tilde{K} \lvert x - y \rvert \leq \tilde{M},
\end{equation} 
where $(\tilde{K},\tilde{M})$ only depends on $(u,\A,f_2)$ through $(L_u,L_\A,\lVert L_{f_2} \rVert_{L^1})$.
\end{lemma}

\begin{proof}
We proceed in several steps. 
\medskip

\emph{\textbf{Step 1.} Regularising parameters.}
For any $\gamma > 0$, we notice that $v : (t,x) \mapsto e^{\gamma t} u(t,x)$ is $C^{1,2}$. 
Moreover, $v(T,x) = e^{\gamma T} g(x)$ and $v$ satisfies
\begin{equation} \label{eq:HJBgamma} 
\partial_t v(t,x) - ( \gamma- c) v(t,x) +  \tfrac{1}{2} \mathrm{Tr} [ \sigma \sigma^\top (t,x) D^2 v (t,x) ] + e^{\gamma t} f_2(t,x) + \inf_{\alpha \in \mathcal{A}} L_\alpha [v] (t,x) = 0,
\end{equation} 
where for every $\alpha$ in $\A$, we define
\begin{equation} \label{eq:alphaOp}
L_\alpha [v] (t,x) := b (t,x,\alpha) \cdot D v (t,x) +  e^{\gamma t} f_1(t,x,\alpha). 
\end{equation} 
To alleviate notations, the dependence on $\gamma$ is not emphasised.
It is equivalent to prove \eqref{eq:Deteriorated} for $v$ instead of $u$.
The growth condition \eqref{eq:DetLin} is satisfied by $v$ with $e^{\gamma T} L_u$ instead of $L_u$.
As a consequence,
\begin{equation} \label{eq:sup5}
\overline{M} := \sup_{t,x,y} v(t,x) - v(t,y) - \varepsilon t^{-1} - \bigg[ \tilde{K} - e^{\gamma T} \int_0^t L_{f_2} (s) \d s \bigg] [ 1 + \lvert x-y \rvert^2 ]^{1/2} - \varepsilon e^{-\delta t} [ \lvert x \rvert^2 + \lvert y \rvert^2 ] 
\end{equation} 
is finite for every $\varepsilon, \delta, \gamma, \tilde{K} > 0$, because of the quadratic term.
To alleviate notations, we do not emphasise the dependence on $\gamma$ for $v$, nor the dependence on $(\varepsilon, \delta, \gamma, \tilde{K})$ for $\overline{M}$. 
The $L_{f_2}$-term within the supremum will allow us to control the $f_2$-term in \eqref{eq:HJBgamma} using only $\lVert L_{f_2} \rVert_{L^1}$.
In the following, $\delta$, $\gamma$ and $\tilde{K}$ are fixed parameters whose values will be chosen later on.
On the contrary, $\varepsilon$ will be sent to $0$, and we restrict ourselves to $\varepsilon \leq 1$.
The $\varepsilon$ terms are regularising terms which will ensure that the supremum is realised within $(0,T) \times \R^d$.
The $e^{-\delta t}$ factor comes from the weak Bernstein method: it will be useful to handle the linear growth of $b$.
Using the continuity and the linear growth of $v$, the quadratic term ensures that the supremum \eqref{eq:sup5} is realised for some $(\overline{t},\overline{x},\overline{y}) \in [0,T] \times \R^d \times \R^d$.
Since $\overline{M}$ is a supremum, it is necessary that $\overline{t} \neq 0$.
In the following, we restrict ourselves to $\tilde{K} \geq e^{\gamma T} \lVert L_{f_2} \rVert_{L^1}$.
\medskip

\emph{\textbf{Step 2.} Bounds on the optimiser.}
If there existed a sequence $(\varepsilon_k)_{k \geq 1}$ of positive parameters which converges to $0$ such that the related sequence $( \overline{t}_k )_{k \geq 1}$ had a constant sub-sequence equal to $T$, then for every $t \in [0,T]$ and $x,y \in \R^d$, we would have (up to re-labelling the sequence)
\begin{multline*}
v(t,x) - v(t,y) - \varepsilon_k t^{-1} - \bigg[ \tilde{K} - e^{\gamma T} \int_0^t L_{f_2} (s) \d s \bigg] [ 1 + \lvert x-y \rvert^2 ]^{1/2} - \varepsilon_k e^{-\delta t} [ \lvert x \rvert^2 + \lvert y \rvert^2 ] \\     
\leq e^{\gamma T} g(\overline{x}) - e^{\gamma T} g(\overline{y}) - \varepsilon_k T^{-1} - \bigg[ \tilde{K} - e^{\gamma T} \int_0^T L_{f_2} (s) \d s \bigg] [ 1 + \lvert \overline{x}-\overline{y} \rvert^2 ]^{1/2} - \varepsilon_k e^{-\delta T} [ \lvert \overline{x} \rvert^2 + \lvert \overline{y} \rvert^2 ],
\end{multline*} 
and using the assumption on $g$ in \ref{ass:Lip}, the r.h.s. is negative as soon as $\tilde{K} \geq e^{\gamma T} ( L_g + \lVert L_{f_2} \rVert_{L^1})$.
Similarly, if there existed $(\varepsilon_k)_{k \geq 1}$ that converges to $0$ while $\overline{M} \leq 0$ for every $k$, then for every $t \in [0,T]$ and $x,y \in \R^d$, we would have
\[ v(t,x) - v(t,y) - \varepsilon_k t^{-1} - \bigg[ \tilde{K} - e^{\gamma T} \int_0^t L_{f_2} (s) \d s \bigg] [ 1 + \lvert x-y \rvert^2 ]^{1/2} - \varepsilon_k e^{-\delta t} [ \lvert x \rvert^2 + \lvert y \rvert^2 ] \leq 0. \]
In both cases, taking the limit would yield
\[ \sup_{t,x,y} v(t,x) - v(t,y) - \bigg[ \tilde{K} - e^{\gamma T} \int_0^t L_{f_2} (s) \d s \bigg] [ 1 + \lvert x-y \rvert^2 ]^{1/2} \leq 0. \]
In particular, this would imply the desired result.

We thus restrict ourselves to $\tilde{K} \geq e^{\gamma T} ( L_g + \lVert L_{f_2} \rVert_{L^1})$, and we can assume without loss of generality that $\overline{t} \neq T$ and $\overline{M} > 0$ for small enough $\varepsilon$. 
From $\overline{M} > 0$ and $L_{f_2} \geq 0$, we get that
\begin{equation} \label{eq:DominK}
[ \tilde{K} - e^{\gamma T} \lVert L_{f_2} \rVert_{L^1} ] \lvert \overline{x} -\overline{y} \rvert \leq v ( \overline{t}, \overline{x} ) - v ( \overline{t}, \overline{y} ).
\end{equation}
Similarly, using the growth condition on $v$,
\[
\varepsilon e^{-\delta \overline{t}} [ \lvert \overline{x} \rvert^2 + \lvert \overline{y} \rvert^2 ] \leq v ( \overline{t}, \overline{x} ) - v ( \overline{t}, \overline{y} ) \leq e^{\gamma T} L_u [ 2 + \lvert \overline{x} \rvert + \lvert \overline{y} \rvert ],
\]
so that, since $\varepsilon \leq 1$,
\[ (\varepsilon \lvert \overline{x} \rvert + \varepsilon \lvert \overline{y} \rvert)^2 \leq 4 e^{(\delta+\gamma)T} L_u + 2 e^{(\delta+\gamma)T} L_u [ \varepsilon \lvert \overline{x} \rvert + \varepsilon \lvert \overline{y} \rvert].  \]
Up to changing $L_u$ in $\max (1,L_u)$, this implies the rough uniform bound
\begin{equation} \label{eq:DetUnifalpha}
\varepsilon \lvert \overline{x} \rvert + \varepsilon \lvert \overline{y} \rvert \leq 2 e^{(\delta+\gamma)T} L_u + 2 e^{(\delta+\gamma)T/2} \sqrt{L_u} \leq 4 e^{(\delta+\gamma)T} L_u,
\end{equation}
which prevents $\varepsilon \lvert \overline{x} \rvert$ and $\varepsilon \lvert \overline{y} \rvert$ from exploding. \medskip

\emph{\textbf{Step 3.} Optimality conditions.}
Let us now define
\[ \overline{p} := \bigg[ \tilde{K} - e^{\gamma T} \int_0^t L_{f_2} (s) \d s \bigg] \left.D_z [1 + \lvert z \rvert^2]^{1/2} \right\rvert_{z=\overline{x}-\overline{y}}. \]
We notice that $\lvert \overline{p} \rvert \leq \tilde{K} - e^{\gamma T} \lVert L_{f_2} \rVert_{L^1} \leq \tilde{K}$.
Since $\overline{t} \notin \{0,T\}$, the first-order optimality conditions in \eqref{eq:sup5} provide
\begin{equation} \label{eq:Det1stOrder}
\begin{cases}
D v( \overline{t}, \overline{x} ) = \overline{p} + 2 \varepsilon e^{-\delta \overline{t}} \overline{x}, \\
D v( \overline{t}, \overline{y} ) = \overline{p} - 2 \varepsilon e^{-\delta \overline{t}} \overline{y}, \\
\partial_t v( \overline{t}, \overline{x} ) - \partial_t v( \overline{t}, \overline{y} ) = -e^{\gamma T} L_{f_2} (\overline{t}) [ 1 + \lvert \overline{x} - \overline{y} \rvert^2 ]^{1/2} -\varepsilon \overline{t}^{-2} - \varepsilon \delta e^{-\delta \overline{t}} [ \lvert \overline{x} \rvert^2 + \lvert \overline{y} \rvert^2 ].
\end{cases}
\end{equation}
We now write the second-order optimality conditions in $(x,y)$ for \eqref{eq:sup5}:
\begin{equation} \label{eq:Det2ndOrder}
\begin{pmatrix}
D^2 v( \overline{t}, \overline{x} ) & 0 \\
0 & - D^2 v( \overline{t}, \overline{y} )
\end{pmatrix} 
\leq \bigg[ \tilde{K} - e^{\gamma T} \int_0^{\overline{t}} L_{f_2} (s) \d s \bigg]
\begin{pmatrix}
\overline{A} & -\overline{A} \\
-\overline{A} & \overline{A}
\end{pmatrix} 
+ 2 \varepsilon e^{-\delta t}
\begin{pmatrix}
\mathrm{Id} & 0 \\
0 & \mathrm{Id}
\end{pmatrix},
\end{equation}
where $\overline{A} := \left.D^2_z [1 + \lvert z \rvert^2]^{1/2} \right\rvert_{z=\overline{x}-\overline{y}}$, the inequality being understood in the sense of quadratic forms. 
The matrix $\overline{A}$ is bounded by a constant $C_0$ independent of $\overline{x}-\overline{y}$.
We multiply the l.h.s. of \eqref{eq:Det2ndOrder} by $
\begin{pmatrix}
\sigma^\top(\overline{t},\overline{x}) &  \sigma^\top(\overline{t},\overline{y})
\end{pmatrix}
$ and the r.h.s. by the transpose of this matrix:
\begin{multline} \label{eq:DetBoundSig}
\sigma^\top(\overline{t},\overline{x}) D^2 v( \overline{t}, \overline{x} ) \sigma(\overline{t},\overline{x}) - \sigma^\top(\overline{t},\overline{y}) D^2 v( \overline{t}, \overline{y} ) \sigma(\overline{t},\overline{y}) \leq 2 \varepsilon e^{-\delta \overline{t}} [ \sigma \sigma^\top (\overline{t},\overline{x}) + \sigma \sigma^\top (\overline{t},\overline{y}) ] \\
+ \bigg[ \tilde{K} - e^{\gamma T} \int_0^{\overline{t}} L_{f_2} (s) \d s \bigg] [ \sigma (\overline{t},\overline{x}) - \sigma (\overline{t},\overline{y}) ]^\top \overline{A} [ \sigma (\overline{t},\overline{x}) - \sigma (\overline{t},\overline{y}) ]. \end{multline} 
We now take the trace and we use the assumption on $\sigma$ in \ref{ass:Lip}.
Since $L_{f_2} \geq 0$, $\sigma \sigma^\top$ has linear growth and $\varepsilon [ \lvert \overline{x} \rvert + \lvert \overline{x} \rvert ]$ is bounded from \eqref{eq:DetUnifalpha}, so that the symmetry property of the trace yields
\begin{equation} \label{eq:DetTrace}
\tfrac{1}{2} \mathrm{Tr}[\sigma\sigma^\top D^2 v (\overline{t},\overline{x}) ] - \tfrac{1}{2} \mathrm{Tr}[\sigma\sigma^\top D^2 v (\overline{t},\overline{y})] \leq C_0 L^2_\sigma \tilde{K} \lvert \overline{x} - \overline{y} ] + C_1 , 
\end{equation} 
for constants $C_0$, $C_1$ which do not depend on $(\varepsilon,\tilde{K},\A)$, and which only depends on $(u,f_2)$ through $(L_u, e^{\gamma T} \lVert L_{f_2} \rVert_{L^1})$. Moreover, $C_0$ does not depend on $\gamma$. \medskip

\emph{\textbf{Step 4.} Gathering the PDE.}
We now write the PDE \eqref{eq:HJBgamma} at $(\overline{t},\overline{x})$ and $(\overline{t},\overline{y})$.
Let $\overline \alpha$ be a minimiser in $\alpha$ of $L_\alpha [ v ] ( \overline{t},\overline{x})$, defined in \eqref{eq:alphaOp}.
Subtracting both PDEs, we get
using \eqref{eq:DetTrace} that 
\begin{multline*}
(\gamma-c) [ v( \overline{t}, \overline{x} ) - v( \overline{t}, \overline{y} ) ] \leq \partial_t v( \overline{t}, \overline{x} ) - \partial_t v( \overline{t}, \overline{y} ) + L_{f_2} ( \overline{t} ) e^{\gamma \overline{t}} [ f_2(\overline{t}, \overline{x}) - f_2(\overline{t}, \overline{y}) ] \\
+ L_{\overline\alpha} [ v ] ( \overline{t}, \overline{x} ) - L_{\overline\alpha} [ v ] ( \overline{t}, \overline{y} ) + C_0 L^2_\sigma \tilde{K} \lvert \overline{x} - \overline{y} ] + C_1.
\end{multline*} 
We then replace $\partial_t v( \overline{t}, \overline{x} ) - \partial_t v( \overline{t}, \overline{y} )$ by its computed value in \eqref{eq:Det1stOrder}.
Assumption \ref{ass:Lip} ensures that the $L_{f_2}$-term in \eqref{eq:Det1stOrder} is greater than $L_{f_2}( \overline{t} ) e^{\gamma \overline{t}} [ f_2(\overline{t}, \overline{x}) - f_2(\overline{t}, \overline{y}) ]$, so that
\begin{multline} \label{eq:DetPDEbefore}
(\gamma-c) [ v( \overline{t}, \overline{x} ) - v( \overline{t}, \overline{y} ) ] \leq -\varepsilon \overline{t}^{-2} - \varepsilon \delta e^{-\delta \overline{t}} [ \lvert \overline{x} \rvert^2 + \lvert \overline{y} \rvert^2 ] \\
+ L_{\overline\alpha} [ v ] ( \overline{t}, \overline{x} ) - L_{\overline\alpha} [ v ] ( \overline{t}, \overline{y} ) + C_0 L^2_\sigma \tilde{K} \lvert \overline{x} - \overline{y} ] + C_1.
\end{multline}
Since $- \varepsilon \overline{t}^{-2} \leq 0$, we can get rid of this term.
To obtain \eqref{eq:defLA}, changing $f_2$ in $e^{\gamma T} f_2$ amounts to replacing $L_b$ by $e^{-\gamma T} L_b$. Since $e^{-\gamma T} \leq 1$, the bound
$\lvert \overline\alpha \rvert \leq L_\A [ 1 + \lvert D v( \overline{t}, \overline{y} ) \rvert]$ still holds. 
Moreover, from \eqref{eq:DetUnifalpha}-\eqref{eq:Det1stOrder},
\[ \lvert D v (\overline{t}, \overline{y} ) \rvert \leq \tilde{K} + 8 e^{(\delta+\gamma)T} L_u \quad \text{   and   } \quad \lvert D v (\overline{t}, \overline{x}) - D v (\overline{t}, \overline{y}) \rvert \leq 2 \varepsilon e^{-\delta \overline{t}} [ \lvert \overline{x} \rvert + \lvert \overline{y} \rvert ]. \] 
We then bound each term of $L_{\overline{\alpha}} [v](\overline{t},\overline{x}) - L_{\overline{\alpha}} [v](\overline{t},\overline{y})$.
For the the $b$-term, we use the Lipschitz assumption \eqref{eq:bipolaire} on $b_2$ and we simply bound $b_1$: 
\begin{align*}
b(\overline{t}, \overline{x},\,&\overline\alpha)\cdot D v( \overline{t}, \overline{x}) - b(\overline{t}, \overline{y}, \overline\alpha) \cdot D v( \overline{t}, \overline{y} ) \\
&\leq 
\lvert D v (\overline{t}, \overline{y} ) \rvert \lvert b(\overline{t}, \overline{x}, \overline\alpha) - b(\overline{t}, \overline{y},\overline\alpha) \rvert + \lvert b (\overline{t}, \overline{x}, \overline\alpha ) \rvert \lvert D v (\overline{t}, \overline{x}) - D v (\overline{t}, \overline{y}) \rvert  \\
&\leq L_b [ \tilde{K} + 8 e^{(\delta+\gamma)T} L_u ] [ 2 + 2 \lvert \overline{\alpha} \rvert + \lvert \overline{x} - \overline{y} \rvert ] +  2 \varepsilon e^{-\delta \overline{t}} L_b [ 2 + \lvert \overline{x} \rvert + \lvert \overline{\alpha} \rvert ] [ \lvert \overline{x} \rvert + \lvert \overline{y} \rvert ], 
\end{align*}
where from \eqref{eq:DetUnifalpha}, $ \varepsilon [ \lvert \overline{x} \rvert + \lvert \overline{y} \rvert ]$ is bounded independently of $\varepsilon$.
For the last term:
\[ \lvert e^{\gamma t} f_1 (\overline{t}, \overline{x},\overline\alpha) - e^{\gamma t} f_1 (\overline{t}, \overline{y},\overline\alpha) \rvert \leq e^{\gamma T} L_{f_2} \lvert \overline{x} - \overline{y} \rvert. \]
We recall that $\varepsilon \leq 1$.
Gathering everything, \eqref{eq:DetPDEbefore} becomes
\begin{multline} \label{eq:DetyNoK}
(\gamma -c)  [ v( \overline{t}, \overline{x} ) - v( \overline{t}, \overline{y} ) ] \leq - \varepsilon \delta e^{-\delta \overline{t}} [ \lvert \overline{x} \rvert^2 + \lvert \overline{y} \rvert^2 ]  + 2 \varepsilon \delta L_b e^{-\delta \overline{t}} \lvert \overline{x} \rvert [ \lvert \overline{x} \rvert + \lvert \overline{y} \rvert ]  + C_2 (\tilde{K} + 1) (\lvert \overline{\alpha} \rvert + 1) \\
+ \lvert \overline{x} - \overline{y} \rvert [ \tilde{K} ( L_b + C_0 L^2_\sigma ) + L_{f_1} e^{\gamma T} + C_3 ], 
\end{multline}
for constants $C_2$, $C_3$ which do not depend on $(\varepsilon, \tilde{K},\A)$, and which only depend on $(u,f_2)$ through $(L_u,\lVert L_{f_2} \rVert_{L^1})$.
Using the bound \eqref{eq:defLA} on $\lvert \overline\alpha \rvert$, we eventually get that
\[
(\gamma - c)  [ v( \overline{t}, \overline{x} ) - v( \overline{t}, \overline{y} ) ] \leq \varepsilon [- \delta + 2 L_b ] e^{-\delta \overline{t}} [ \lvert \overline{x} \rvert^2 + \lvert \overline{y} \rvert^2 ] + C_4 + \lvert \overline{x} - \overline{y} \rvert [ \tilde{K} ( L_b + C_0 L^2_\sigma ) + C_5 ], \] 
for constants $C_4, C_5$ which do not depend on $\varepsilon$, and which only depend on $(u,\A,f_2)$ through $(L_u,L_\A,\lVert L_{f_2} \rVert_{L^1})$.\medskip

\emph{\textbf{Step 5.} Choice of parameters.}
Let us fix the values of parameters:
\[ \delta := 2L_b, \quad \gamma := c + L_b + C_0 L^2_\sigma + 2, \quad \tilde{K} := \max ( e^{\gamma T} L_g + e^{\gamma T} \lVert L_{f_2} \rVert_{L^1}, C_5 + e^{\gamma T} \lVert L_{f_2} \rVert_{L^1} ), \]
where we recall that $C_0$ was not depending on anything. This yields
\[ (\gamma - c) [ v( \overline{t}, \overline{x} ) - v( \overline{t}, \overline{y} ) ] \leq [\gamma - c - 1] [\tilde{K} - e^{\gamma T} \lVert L_{f_2} \rVert_{L^1}] \lvert \overline{x} - \overline{y} \rvert  + C_4,  \]
so that, using \eqref{eq:DominK},
\[ \overline{M} \leq v( \overline{t}, \overline{x} ) - v( \overline{t}, \overline{y} ) \leq C_4. \]
This uniform bound on $\overline{M}$ allows us to send $\varepsilon$ to $0$ to get that
\[ \sup_{t,x,y} v(t,x) - v(t,y) - \bigg[ \tilde{K} - e^{\gamma T} \int_0^t L_{f_2} (s) \d s \bigg] [ 1 + \lvert x-y \rvert^2 ]^{1/2} \leq C_4. \]
We conclude the proof by choosing $\tilde{M} := \tilde{K} + C_4$.
\end{proof}

We are now ready to prove the Lipschitz estimate using the Ishii-Lions method. 

\begin{proposition}[Lipschitz estimate] \label{pro:EstimLip} \label{pro:Lip}
Let $u$ in $C([0,T] \times \R^d) \cap C^{1,2}((0,T) \times \R^d)$ be a solution of \eqref{eq:HJB} and $L_u > 0$ such that for every $(t,x) \in [0,T] \times \R^d$,
\begin{equation*}
\lvert u(t,x) \rvert \leq L_u [1 +\lvert x\rvert]. 
\end{equation*} 
Under \ref{ass:Lip}-\ref{ass:Boundsigma}, there exists a positive constant $K > 0$ such that
\begin{equation} \label{eq:Lip}
\forall t \in [0,T], \forall (x,y) \in \R^d \times \R^d, \quad \lvert u(t,x) - u(t,y) \rvert \leq K \vert x - y \rvert,
\end{equation} 
where $K$ only depends on $(u,\A,f_2)$ through $(L_u,L_\A,\lVert L_{f_2} \rVert_{L^1})$.
\end{proposition}

\begin{proof}
As previously, we look for $K > 0$ such that  
\begin{equation*} \label{eq:goal}
\sup_{t,x,y} u(t,x) - v(t,y) - K \lvert x - y \rvert \leq 0. 
\end{equation*}
We first introduce regularising parameters.
\medskip

\emph{\textbf{Step 1.} Regularising parameters.}
As in the proof of Lemma \ref{lem:detoriated},
for any $\gamma > 0$, we introduce the $C^{1,2}$ solution $v : (t,x) \mapsto e^{\gamma t} u(t,x)$ of \eqref{eq:HJBgamma}.
It is equivalent to prove \eqref{eq:goal} for $v$ instead of $u$, and $v$ has linear growth with constant $e^{\gamma T} L_u$.
To make use of uniform ellipticity, we consider a smooth function $\rho : \R_+ \rightarrow [0,1]$ such that
\begin{equation*}
\begin{cases}
\rho(z) = \tfrac{1}{3} z^{3/2} &\text{ for } z \in [0,1], \\
\rho(z) = 0 &\text{ for } z \geq 2,
\end{cases}
\end{equation*}
and we define the non-negative function $\psi(z) := z - \rho(z)$.
As previously, we introduce regularising parameters $\varepsilon, \delta, \varepsilon > 0$, before considering the supremum
\begin{equation} \label{eq:sup}
\overline{M} := \sup_{t,x,y} v(t,x) - v(t,y) - \bigg[ K - e^{\gamma T} \int_0^t L_{f_2} (s) \d s \bigg] \psi (\lvert x-y \rvert)  - \varepsilon t^{-1} - \varepsilon e^{-\delta t} [ \lvert x \rvert^2 + \lvert y \rvert^2 ]. 
\end{equation} 
As previously, $\gamma$, $\delta$ and $K$ are fixed parameters whose values will be chosen later on.
On the contrary, $\varepsilon$ will be sent to $0$, and we restrict ourselves to $\varepsilon \leq 1$.
The $L_{f_2}$-term within the supremum will allow us to control the $f_2$-term in \eqref{eq:HJBgamma} using only $\lVert L_{f_2} \rVert_{L^1}$.
In the following, we restrict ourselves to $K \geq e^{\gamma T} \lVert L_{f_2} \rVert_{L^1}$.

The quadratic term ensures that the supremum \eqref{eq:sup} is realised for some $(\overline{t},\overline{x},\overline{y}) \in [0,T] \times \R^d \times \R^d$ with $\overline{t} \neq 0$. If there existed a sequence $(\varepsilon_k)_{k \geq 1}$ of positive parameters which converges to $0$ such that the related sequence $( \overline{t}_k )_{k \geq 1}$ had a constant sub-sequence equal to $T$, we would get the desired result as soon as $K \geq e^{\gamma T} L_g + e^{\gamma T} \lVert L_{f_2} \rVert_{L^1}$, reasoning as in \emph{\textbf{Step 2.}} in the proof of Lemma \ref{lem:detoriated}.
Similarly, the result would be direct if there existed $(\varepsilon_k)_{k \geq 1}$ which converges to $0$ while $\overline{M} \leq 0$ for every $k$. \medskip

\emph{\textbf{Step 2.} Bounds on the optimiser.}
We thus restrict ourselves to $K \geq e^{\gamma T} L_g + e^{\gamma T} \lVert L_{f_2} \rVert_{L^1}$ and we reason by contradiction, assuming that $\overline{t} \neq T$ and $\overline{M} > 0$ for $\varepsilon$ small enough:
we are going to show that this cannot happen if $K$ is larger than a certain threshold which does not depend on $\varepsilon$.
From Lemma \ref{lem:detoriated}, we deduce that there exist large enough constant $\tilde{K}, \tilde{M} > 0$ such that 
\[ \sup_{t,x,y} v ( t,x ) - v ( t,y) - \tilde{K} \lvert x - y \rvert \leq \tilde{M}, \]
where $(\tilde{K},\tilde{M})$ only depends on $(u,\A,f_2)$ through $(L_u,L_\A,\lVert L_{f_2} \rVert_{L^1})$.
From $\overline{M} > 0$, we get that
\[
\bigg[ K - e^{\gamma T} \int_0^t L_{f_2} (s) \d s - \tilde{K} \bigg] [ \lvert \overline{x} -\overline{y} \rvert - \rho( \lvert \overline{x} -\overline{y} \rvert ) ] \leq v ( \overline{t}, \overline{x} ) - v ( \overline{t}, \overline{y} ) - \tilde{K} [ \lvert \overline{x} -\overline{y} \rvert - \rho( \lvert \overline{x} -\overline{y} \rvert ) ].
\]
For every $z \geq 0$, we have $\tfrac{1}{2}z - \rho(z) \geq 0$ and $\rho(z) \leq 1$, hence
\[
\tfrac{1}{2} \bigg[ K - e^{\gamma T} \int_0^t L_{f_2} (s) \d s - \tilde{K} \bigg] \lvert \overline{x} -\overline{y} \rvert \leq  \tilde{K} + \sup_{t,x,y} v ( t,x ) - v ( t,y) - \tilde{K} \lvert x - y \rvert \leq \tilde{K} + \tilde{M}.
\]
Up to increasing the value of $\tilde{K}$, we can thus assume that for $K \geq \tilde{K}$,
\begin{equation} \label{eq:DominDiff}
\lvert \overline{x} -\overline{y} \rvert \leq C_0 K^{-1},
\end{equation} 
for some $C_0>0$ which only depends on $(e^{\gamma T},\tilde{M},\tilde{K},\lVert L_{f_2} \rVert_{L^1})$.
Up to increasing $\tilde{K}$ again, we can then assume that $\lvert \overline{x} -\overline{y} \rvert \leq 1$ for $K \geq \tilde{K}$.
Using the linear growth of $v$ and $\overline{M} > 0$,
\[
\varepsilon e^{-\delta \overline{t}} [ \lvert \overline{x} \rvert^2 + \lvert \overline{y} \rvert^2 ] \leq v ( \overline{t}, \overline{x} ) - v ( \overline{t}, \overline{y} ) \leq L_u e^{\gamma T} [ 2 + \lvert \overline{x} \rvert + \lvert \overline{y} \rvert ],
\]
and following \emph{\textbf{Step 2.}} in the proof of Lemma \ref{lem:detoriated}, this implies that
\begin{equation} \label{eq:Unifalpha}
\varepsilon \lvert \overline{x} \rvert + \varepsilon \lvert \overline{y} \rvert \leq 4 e^{(\delta + \gamma) T} L_u.
\end{equation}
This bound prevents $\varepsilon \lvert \overline{x} \rvert$ and $\varepsilon \lvert \overline{y} \rvert$ from exploding. \medskip

\emph{\textbf{Step 3.} Optimality conditions.}
To ease computations, we introduce the function $\chi : z \in \R^d \mapsto \lvert z \rvert.$
Since $\overline{t} \notin \{ 0, T \}$ and $\overline{x} \neq \overline{y}$, the first-order optimality conditions in \eqref{eq:sup} provide
\begin{equation} \label{eq:1stOrder}
\begin{cases}
D v( \overline{t}, \overline{x} ) = \big[ K - e^{\gamma T} \int_0^t L_{f_2} (s) \d s \big] \psi' ( \lvert \overline{x} - \overline{y} \rvert ) D \chi ( \overline{x} - \overline{y} ) + 2 \varepsilon e^{-\delta \overline{t}} \overline{x} , \\
D v( \overline{t}, \overline{y} ) = \big[ K - e^{\gamma T} \int_0^t L_{f_2} (s) \d s \big] \psi' ( \lvert \overline{x} - \overline{y} \rvert ) D \chi ( \overline{x} - \overline{y} ) - 2 \varepsilon e^{-\delta \overline{t}} \overline{y}, \\
\partial_t v( \overline{t}, \overline{x} ) - \partial_t v( \overline{t}, \overline{y} ) = -e^{\gamma T} L_{f_2} (\overline{t}) [ 1 + \lvert \overline{x} - \overline{y} \rvert^2 ]^{1/2} -\varepsilon \overline{t}^{-2} - \varepsilon \delta e^{-\delta \overline{t}} [ \lvert \overline{x} \rvert^2 + \lvert \overline{y} \rvert^2 ].
\end{cases}
\end{equation}
Since $\lvert \overline{x} - \overline{y} \rvert \leq 1$, we notice that $0 \leq \psi' ( \lvert \overline{x} - \overline{y} \rvert ) \leq 1$.
We then write the second-order optimality conditions in $(x,y)$ for \eqref{eq:sup}:
\begin{equation} \label{eq:2ndOrder}
\begin{pmatrix}
D^2 v( \overline{t}, \overline{x} ) & 0 \\
0 & - D^2 v( \overline{t}, \overline{y} )
\end{pmatrix} 
\leq 
\bigg[ K - e^{\gamma T} \int_0^t L_{f_2} (s) \d s \bigg]
\begin{pmatrix}
\overline{A} & -\overline{A} \\
-\overline{A} & \overline{A}
\end{pmatrix} 
+ 2 \varepsilon e^{-\delta \overline{t}}
\begin{pmatrix}
\mathrm{Id} & 0 \\
0 & \mathrm{Id}
\end{pmatrix},
\end{equation}
where 
\[ \overline{A} := \psi''(\lvert \overline{x} - \overline{y} \rvert) D \chi ( \overline{x} - \overline{y} ) \otimes D \chi ( \overline{x} - \overline{y} ) + \psi'(\lvert \overline{x} - \overline{y} \rvert) D^2 \chi ( \overline{x} - \overline{y} ).  \]
For every $(p,q) \in \R^d \times \R^d$, applying \eqref{eq:2ndOrder} to the vector $\begin{pmatrix}
p & q   
\end{pmatrix}$ and using $L_{f_2} \geq 0$, we get that 
\begin{equation} \label{eq:2ndOrderDev}
p^\top D^2 v( \overline{t}, \overline{x} ) p - q^\top D^2 v( \overline{t}, \overline{y} ) q \leq K (p-q)^\top \overline{A} (p-q) +  2 \varepsilon e^{-\delta \overline{t}} [ \lvert p \rvert^2 +\lvert q \rvert^2 ]. 
\end{equation}
Since $\lvert D \chi (z) \rvert^2 = 1$ for every $z$, we get that $D^2 \chi( z ) D \chi (z) = 0$.
Using \eqref{eq:2ndOrderDev} with $p = - q = D \chi (\overline{x} - \overline{y})$ thus yields
\begin{equation*} 
 D \chi (\overline{x} - \overline{y})^\top [ D^2 v( \overline{t}, \overline{x} ) - D^2 v( \overline{t}, \overline{y} ) ]  D \chi (\overline{x} - \overline{y}) \leq 4 K \psi''(\lvert \overline{x} - \overline{y} \rvert) +  4 \varepsilon e^{-\delta \overline{t}}. 
\end{equation*}
Similarly, taking $p = q$ in \eqref{eq:2ndOrderDev} gives that
\begin{equation} \label{eq:Controlvp}
\forall p \in \R^d, \quad p^\top [ D^2 v( \overline{t}, \overline{x} ) - D^2 v( \overline{t}, \overline{y} ) ] p \leq 4 \varepsilon e^{-\delta \overline{t}} \lvert p \rvert^2. 
\end{equation}
Completing $D \chi (\overline{x} - \overline{y})$ in an orthornomal basis $(D \chi (\overline{x} - \overline{y}),p_2, \ldots, p_d)$ of $\R^d$, this proves that
\begin{equation} \label{eq:ExplodeTr}
\begin{split}
\mathrm{Tr}[ D^2 v( \overline{t}, \overline{x} ) - D^2 v( \overline{t}, \overline{y} ) ] &= D \chi (\overline{x} - \overline{y})^\top [ D^2 v( \overline{t}, \overline{x} ) - D^2 v( \overline{t}, \overline{y} ) ]  D \chi (\overline{x} - \overline{y}) \\ 
&\phantom{abcdefghijklmnopqrstu.}+ \sum_{i=2}^d p_i^\top [ D^2 v( \overline{t}, \overline{x} ) - D^2 v( \overline{t}, \overline{y} ) ] p_i \\
&\leq 4 K \psi''(\lvert \overline{x} - \overline{y} \rvert) +  4 \varepsilon d e^{-\delta \overline{t}}.   
\end{split}
\end{equation}
As $K \rightarrow + \infty$, $\lvert \overline{x} - \overline{y} \rvert \rightarrow 0$ from \eqref{eq:DominDiff}, and our specific choice of $\psi$ makes the r.h.s. of \eqref{eq:ExplodeTr} go to $-\infty$: this is the main trick of the Ishii-Lions method. \medskip

\emph{\textbf{Step 4.} Control of the second-order terms.}
Let us introduce the matrix $a := \sigma \sigma^\top$. 
From \ref{ass:Boundsigma}, $(t,x) \mapsto a(t,x)$ is uniformly continuous with modulus $m_\sigma$.
We decompose:
\begin{multline} \label{eq:Decomp}
\mathrm{Tr}[a(\overline{t}, \overline{x} ) D^2 v( \overline{t}, \overline{x} ) - a(\overline{t}, \overline{y} ) D^2 v( \overline{t}, \overline{y} ) ] = \tfrac{1}{2} \mathrm{Tr}\big\{ a(\overline{t}, \overline{x} ) [ D^2 v( \overline{t}, \overline{x} ) - D^2 v( \overline{t}, \overline{y} ) ] \big\} \\ 
+ \tfrac{1}{2} \mathrm{Tr} \big\{a(\overline{t}, \overline{y} ) [ D^2 v( \overline{t}, \overline{x} ) - D^2 v( \overline{t}, \overline{y} )] \big\} + \tfrac{1}{2} \mathrm{Tr} \big\{ [a(\overline{t}, \overline{x} ) - a(\overline{t}, \overline{y} ) ] [ D^2 v( \overline{t}, \overline{x} ) + D^2 v( \overline{t}, \overline{y} ) ] \big\}.
\end{multline} 
We now apply Lemma \ref{lem:CompTr} in Appendix with $A = a(\overline{t}, \overline{x} )$, $B = D^2 v( \overline{t}, \overline{x} ) - D^2 v( \overline{t}, \overline{y} )$, $m = \eta_\sigma / 2$ and $M = 4 \varepsilon e^{-\delta \overline{t}}$.
The inequality $A \geq m \mathrm{Id}$ stems from \ref{ass:Boundsigma}, while $B \leq M \mathrm{Id}$ results from \eqref{eq:Controlvp}.
This gives that
\begin{equation*}
\mathrm{Tr}\big\{ a(\overline{t}, \overline{x} ) [ D^2 v( \overline{t}, \overline{x} ) - D^2 v( \overline{t}, \overline{y} ) ] \big\} \leq \tfrac{\eta_\sigma}{2} \mathrm{Tr}[ D^2 v( \overline{t}, \overline{x} ) - D^2 v( \overline{t}, \overline{y} ) ] + 4 \varepsilon e^{-\delta \overline{t}} \mathrm{Tr}[ a(\overline{t}, \overline{x} ) ],
\end{equation*}
and $\mathrm{Tr}[ a(\overline{t}, \overline{x} ) ]$ has linear growth from \ref{ass:Lip}.
Using \eqref{eq:Unifalpha}, $4 \varepsilon e^{-\delta \overline{t}} \mathrm{Tr}[ a(\overline{t}, \overline{x} ) ]$ is bounded by a constant $C_1$ which only depends on $(L_\sigma,T,\delta,L_u)$. The same bound works for $\mathrm{Tr} \{ a(\overline{t}, \overline{y} ) [ D^2 v( \overline{t}, \overline{x} ) - D^2 v( \overline{t}, \overline{y} ) ] \}$.
We then use that $A \mapsto \sup_{\lvert \xi \rvert = 1} \xi^\top A \xi$ is a sub-multiplicative norm on symmetric matrices and that all norms are equivalent in finite dimension.
By \ref{ass:Lip} and the continuity of the trace, there exists a constant $C_2 > 0$ which only depends on the dimension $d$ such that
\begin{equation} \label{eq:SorcInt}
\mathrm{Tr} \big\{ [a(\overline{t}, \overline{x} ) - a(\overline{t}, \overline{y} ) ] [ D^2 v( \overline{t}, \overline{x} ) + D^2 v( \overline{t}, \overline{y} ) ] \big\} \leq C_2 m_\sigma( \lvert \overline{x} - \overline{y} \rvert ) \sup_{\lvert \xi \rvert =1} \xi^\top [ D^2 v( \overline{t}, \overline{x} ) + D^2 v( \overline{t}, \overline{y} ) ] \xi. 
\end{equation}
Using \eqref{eq:2ndOrder}, we now apply Lemma \ref{lem:CompTrSorcier} with $A = \overline{A}$, $X = D^2 v( \overline{t}, \overline{x} )$, $Y = -D^2 v( \overline{t}, \overline{x} )$, $m = 2 \varepsilon e^{-\delta \overline{t}}$ and $C = K - e^{\gamma T} \int_0^t L_{f_2} (s) \d s$.
As a consequence, the r.h.s. of \eqref{eq:SorcInt} is bounded by
\[ \sqrt{2} C_2 m_\sigma( \lvert \overline{x} - \overline{y} \rvert ) \bigg\{ 4 \varepsilon e^{-\delta \overline{t}} + 2 \bigg[ K - e^{\gamma T} \int_0^t L_{f_2} (s) \d s \bigg] + \sup_{\lvert \xi \rvert = 1}  \xi^\top [ D^2 v( \overline{t}, \overline{x} ) - D^2 v( \overline{t}, \overline{y} ) ] \xi \bigg\}. \] 
The supremum that appears within the above quantity equals the spectral radius of the symmetric matrix $D^2 u( \overline{t}, \overline{x} ) - D^2 u( \overline{t}, \overline{y} )$.
From \eqref{eq:Controlvp}, any eigenvalue of $D^2 v( \overline{t}, \overline{x} ) - D^2 v( \overline{t}, \overline{y} )$ is bounded from above by $4 \varepsilon e^{-\delta \overline{t}}$. 
As a consequence, any eigenvalue of $D^2 v( \overline{t}, \overline{x} ) - D^2 v( \overline{t}, \overline{y} )$ is bounded from below by
\[ \mathrm{Tr}[D^2 v( \overline{t}, \overline{x} ) - D^2 v( \overline{t}, \overline{y} )] - 4 (d-1) \varepsilon e^{-\delta \overline{t}}, \]
so that
\[ \sup_{\lvert \xi \rvert = 1}  \xi^\top [ D^2 v( \overline{t}, \overline{x} ) - D^2 v( \overline{t}, \overline{y} ) ] \xi \leq \lvert \mathrm{Tr}[D^2 v( \overline{t}, \overline{x} ) - D^2 v( \overline{t}, \overline{y} )] \rvert + 4 d \varepsilon e^{-\delta \overline{t}}. \]
At the end of the day, gathering everything from \eqref{eq:Decomp} yields
\begin{multline} \label{eq:DomTr}
\mathrm{Tr}[a(\overline{t}, \overline{x} ) D^2 v( \overline{t}, \overline{x} ) - a(\overline{t}, \overline{y} ) D^2 v( \overline{t}, \overline{y} ) ] \leq \tfrac{\eta_\sigma}{2} \mathrm{Tr}[ D^2 v( \overline{t}, \overline{x} ) - D^2 v( \overline{t}, \overline{y} ) ] + C_1 \\
+ \sqrt{2} C_2 m_\sigma( \lvert \overline{x} - \overline{y} \rvert ) \big[ 2 \varepsilon e^{-\delta \overline{t}} + K + 2 d \varepsilon e^{-\delta \overline{t}} + \tfrac{1}{2} \lvert \mathrm{Tr}[D^2 v( \overline{t}, \overline{x} ) - D^2 v( \overline{t}, \overline{y} )] \rvert \big].
\end{multline}
Using \eqref{eq:ExplodeTr}, we see that the r.h.s. of \eqref{eq:DomTr} goes to $-\infty$ as $K \rightarrow +\infty$. \medskip

\emph{\textbf{Step 5.} Gathering the PDE.}
We now write the PDE \eqref{eq:HJBgamma} at $(\overline{t},\overline{x})$ and $(\overline{t},\overline{y})$.
Let $\overline \alpha$ be a minimiser in $\alpha$ of $L_\alpha [ v ] ( \overline{t},\overline{x})$, defined in \eqref{eq:alphaOp}.
Subtracting both PDEs, we get
that 
\begin{multline*}
(\gamma -c) [ v( \overline{t}, \overline{x} ) - v( \overline{t}, \overline{y} ) ] = \partial_t v( \overline{t}, \overline{x} ) - \partial_t v( \overline{t}, \overline{y} ) + L_{f_2} ( \overline{t} ) e^{\gamma \overline{t}} [ f_2(\overline{t}, \overline{x}) - f_2(\overline{t}, \overline{y}) ] \\
+ L_{\overline\alpha} [ v ] ( \overline{t}, \overline{x} ) - L_{\overline\alpha} [ v ] ( \overline{t}, \overline{y} ) + \tfrac{1}{2} \mathrm{Tr}[a(\overline{t}, \overline{x} ) D^2 v( \overline{t}, \overline{x} ) - a(\overline{t}, \overline{y} ) D^2 v( \overline{t}, \overline{y} ) ].
\end{multline*} 
Using \eqref{eq:Unifalpha}-\eqref{eq:1stOrder}, we now reason as in \emph{\textbf{Step 4.}} in the proof of Lemma \ref{lem:detoriated}: we replace $\partial_t v( \overline{t}, \overline{x} ) - \partial_t v( \overline{t}, \overline{y} )$ by its computed value in \eqref{eq:1stOrder}, we use it to get rid of $L_{f_2} ( \overline{t} ) e^{\gamma \overline{t}} [ f_2(\overline{t}, \overline{x}) - f_2(\overline{t}, \overline{y}) ]$, we get rid of the $- \varepsilon \overline{t}^{-2}$ term, and then we control the $L_{\overline\alpha} [ v ] ( \overline{t}, \overline{x} ) - L_{\overline\alpha} [ v ] ( \overline{t}, \overline{y} )$ term.
For the the $b$-term, we now use the Lipschitz assumption \eqref{eq:bipolaire} on $b_1$ and $b_2$:
\begin{align*}
b(\overline{t}, \overline{x},\,&\overline\alpha)\cdot D v( \overline{t}, \overline{x}) - b(\overline{t}, \overline{y}, \overline\alpha) \cdot D v( \overline{t}, \overline{y} ) \\
&\leq 
\lvert D v (\overline{t}, \overline{y} ) \rvert \lvert b(\overline{t}, \overline{x}, \overline\alpha) - b(\overline{t}, \overline{y},\overline\alpha) \rvert + \lvert b (\overline{t}, \overline{x}, \overline\alpha ) \rvert \lvert D v (\overline{t}, \overline{x}) - D v (\overline{t}, \overline{y}) \rvert  \\
&\leq L_b [ \tilde{K} + 8 e^{(\delta+\gamma)T} L_u ] ( 1 + \lvert \overline{\alpha} \rvert ) \lvert \overline{x} - \overline{y} \rvert +  2 \varepsilon e^{-\delta \overline{t}} L_b [ 1 + \lvert \overline{x} \rvert + \lvert \overline{\alpha} \rvert ] [ \lvert \overline{x} \rvert + \lvert \overline{y} \rvert ], 
\end{align*}
where from \eqref{eq:Unifalpha}, $ \varepsilon [ \lvert \overline{x} \rvert + \lvert \overline{y} \rvert ]$ is bounded independently of $\varepsilon$. 
Similarly,
\[ \lvert e^{\gamma t} f_1 (\overline{t},\overline{x},\overline{\alpha} ) - f_1 ( \overline{t}, \overline{y}, \overline{\alpha} ) \vert \leq e^{\gamma T} L_{f_2} \lvert \overline{x} - \overline{y} \vert. \]
As in \emph{\textbf{Step 4.}} in the proof of Lemma \ref{lem:detoriated}, we obtain using \eqref{eq:defLA} that
$\lvert \overline\alpha \rvert \leq L_\A [ 1 + \lvert D v( \overline{t}, \overline{y} ) \rvert]$.
Gathering everything, the analogous of \eqref{eq:DetyNoK} now reads
\begin{multline*}
(\gamma -c)  [ v( \overline{t}, \overline{x} ) - v( \overline{t}, \overline{y} ) ] \leq \varepsilon [2 L_b - \delta] \delta e^{-\delta \overline{t}} [ \lvert \overline{x} \rvert^2 + \lvert \overline{y} \rvert^2 ] + \tfrac{1}{2} \mathrm{Tr}[a(\overline{t}, \overline{x} ) D^2 v( \overline{t}, \overline{x} ) - a(\overline{t}, \overline{y} ) D^2 v( \overline{t}, \overline{y} ) ] \\
+ L_b [ \tilde{K} + 8 e^{(\delta+\gamma)T} L_u ] ( 1 + K ) \lvert \overline{x} - \overline{y} \rvert + C_3 ( 1 + K + \lvert \overline{x} - \overline{y} \rvert ),  
\end{multline*}
for a constant $C_3$ which does not depend on $(\varepsilon, K)$, and which only depends on $(u,\A,f_2)$ through $(L_u,L_\A,\lVert L_{f_2} \rVert_{L^1})$.\medskip

\emph{\textbf{Step 6.} Contradiction.}
We now fix the values $\gamma := c$ and $\delta := 2 L_b$.
Using \eqref{eq:DominDiff} and imposing that $K \geq 1$, we get that
\begin{equation} \label{eq:DomPDEFinal}
0 \leq \tfrac{1}{2} \mathrm{Tr}[a(\overline{t}, \overline{x} ) D^2 v( \overline{t}, \overline{x} ) - a(\overline{t}, \overline{y} ) D^2 v( \overline{t}, \overline{y} ) ] + 2 C_0 L_b [ K + 8 e^{(\delta+\gamma)T} L_u ] + 2 C_3 ( 1 + K + C_0 ).
\end{equation}
Similarly, \eqref{eq:DomTr} becomes
\begin{multline} \label{eq:DomTrFinal}
\mathrm{Tr}[a(\overline{t}, \overline{x} ) D^2 v( \overline{t}, \overline{x} ) - a(\overline{t}, \overline{y} ) D^2 v( \overline{t}, \overline{y} ) ] \leq \tfrac{\eta_\sigma}{2} \mathrm{Tr}[ D^2 v( \overline{t}, \overline{x} ) - D^2 v( \overline{t}, \overline{y} ) ] + C_1 \\
+ \sqrt{2} C_2 m_\sigma( \lvert \overline{x} - \overline{y} \rvert ) \big[ 2 + K + 2 d + \tfrac{1}{2} \lvert \mathrm{Tr}[D^2 v( \overline{t}, \overline{x} ) - D^2 v( \overline{t}, \overline{y} )] \rvert \big],
\end{multline}
and $m_\sigma( \lvert \overline{x} - \overline{y} \rvert ) \rightarrow 0$ as $K \rightarrow +\infty$.
From \eqref{eq:DominDiff}-\eqref{eq:ExplodeTr},
\begin{equation*} \label{eq:ExplodeTrFinal}
\mathrm{Tr}[ D^2 v( \overline{t}, \overline{x} ) - D^2 v( \overline{t}, \overline{y} ) ] \leq -3 \sqrt{C_0} K^{3/2} +  4 d,
\end{equation*}
proving that $\mathrm{Tr}[ D^2 v( \overline{t}, \overline{x} ) - D^2 v( \overline{t}, \overline{y} ) ]$ goes to $-\infty$ as $K \rightarrow + \infty$ uniformly in $\varepsilon$. 
From \eqref{eq:DomTrFinal}, $\mathrm{Tr}[a(\overline{t}, \overline{x} ) D^2 v( \overline{t}, \overline{x} ) - a(\overline{t}, \overline{y} ) D^2 v( \overline{t}, \overline{y} ) ]$ goes to $-\infty$ at least as fast as $-K^{3/2}$ as $K \rightarrow + \infty$, uniformly in $\varepsilon$.
From this, there exists a finite value $K' \geq \max (1,\tilde{K},e^{\gamma T}(  L_g + \lVert L_{f_2} \rVert_{L^1} ))$ independent of $\varepsilon$ such that the r.h.s. of \eqref{eq:DomPDEFinal} is always negative for $K \geq K '$. 
This gives the desired contradiction. 
We moreover notice that $K'$ only depends on $(u,\A,f_2)$ through $(L_u,L_\A,\lVert L_{f_2} \rVert_{L^1})$.
\end{proof}

\begin{proof}[Proof of Theorem \ref{thm:Main}]
From Lemma \ref{lem:Growth}, there exists a growth constant $L$ which is valid for any $C^{1,2}$ linear growth solution of \eqref{eq:HJB}, and $L$ only depends on $\A$ through $L_\A$.
Let $K$ be the value given by Proposition \ref{pro:Lip} for $L_u = L$: $K$ only depends on $\A$ through $L_\A$.
$K$ can thus be chosen as an increasing function of $L_\A$. 
Up to choosing a larger $L_\A$ in \eqref{eq:defLA}, we can assume that $L_\A [ 1 +K ] \geq R_\A$, where $R_\A$ is the constant given by \ref{ass:Holder}. 
Let us define the compact sub-set of $\A$:
\[ \A' := \A \cap B(0, L_\A [ 1 + L +K ]). \]
We now introduce a truncated version of \eqref{eq:HJB}: 
we truncate the non-linearity by using the compact control set $\A'$ instead of $\A$. 
From \cite[Theorem 3.1]{rubio2011existence}, the resulting truncated equation has a unique linear growth solution $\tilde{u}$ in $C([0,T] \times \R^d) \cap C^{1,2,\beta,\beta}_\mathrm{loc}$.
From the definition of $\A'$, \eqref{eq:defLA} is still valid when replacing $\A$ by $\A'$, with the same constant $L_\A$.
Since the Lipschitz estimate of Proposition \ref{pro:EstimLip} only depends on $\A$ through $L_\A$, $\tilde{u}$ is $K$-Lipschitz continuous with the same constant $K$. 
This proves that $\tilde{u}$ does not feel the truncation, hence $\tilde{u}$ is a $C^{1,2}$ solution of \eqref{eq:HJB}.

Reciprocally, any linear growth solution in $C([0,T] \times \R^d) \cap C^{1,2,\beta,\beta}_\mathrm{loc}$
of \eqref{eq:HJB} is $K$-Lipschitz continuous from Lemma \ref{lem:Growth} and Proposition \ref{pro:EstimLip}, hence it is a solution of the truncated equation. 
Since uniqueness holds for the truncated equation, this concludes the proof.
\end{proof}

\appendix

\section{Linear algebra results} \label{sec:app}

This appendix gathers two linear algebra results that were needed for \emph{\textbf{Step 4.}} in the proof of Proposition \ref{pro:EstimLip}.
These lemmas correspond to intermediary results which can be respectively found within the proofs of \cite[Proposition III.1]{ishii1990viscosity} and \cite[Lemma III.1]{ishii1990viscosity}.
We provide the proofs for the sake of completeness.

\begin{lemma} \label{lem:CompTr}
Let $A$ and $B$ be two symmetric matrices in $\R^d$ with $A \geq m \mathrm{Id}$ and $B \leq M \mathrm{Id}$ for real numbers $m, M \geq 0$. 
We have that
\[ \mathrm{Tr}[A B] \leq m \mathrm{Tr}[ B ] + M [ \mathrm{Tr}[ A ] - d m ]. \]
\end{lemma}

\begin{proof}
We first handle the case $M = 0$.
Let us write $A = P^\top D P$ for some orthogonal matrix $P$ and $D := \mathrm{diag}(d_1,\ldots,d_d)$. 
The $d_i$ are the eigenvalues of $A$ and all satisfy $d_i \geq m$.
Moreover, using the symmetry property of the trace,
\[ \mathrm{Tr}[A B] = \mathrm{Tr}[D (P B P^\top)  ] = \sum_{i=1}^d d_i e_i^\top ( P B P^\top ) e_i, \]
where $(e_1,\ldots,e_d)$ is an orthonormal basis of $\R^d$.
From the non-positiveness assumption $M \leq 0$ on $B$, we have $e_i^\top ( P B P^\top ) e_i \leq 0$, so that 
\[ \mathrm{Tr}[A B] \leq m \sum_{i=1}^d e_i^\top ( P B P^\top ) e_i = m \mathrm{Tr}[ P B P^\top  ] = m \mathrm{Tr}[B]. \]
In the general case $B$ does not satisfy the non-positiveness condition, but $B - M \mathrm{Id}$ does.
The above result applied to $B - M \mathrm{Id}$ then yields
\[ \mathrm{Tr}[A B] - M \mathrm{Tr}[A] \leq m \mathrm{Tr}[B] - m M \mathrm{Tr}[\mathrm{Id}] =  m \mathrm{Tr}[B] - d m M, \]
concluding the proof.
\end{proof}

\begin{lemma} \label{lem:CompTrSorcier}
Let $A$, $X$ and $Y$ be symmetric matrices in $\R^d$ such that 
\begin{equation} \label{eq:SorceComp}
\begin{pmatrix}
X & 0 \\
0 & Y
\end{pmatrix} 
\leq
C
\begin{pmatrix}
A & -A \\
-A & A 
\end{pmatrix}
+
m
\begin{pmatrix}
\mathrm{Id} & 0 \\
0 & \mathrm{Id}
\end{pmatrix},
\end{equation}    
for real numbers $C, m > 0$. Then
\[ \sup_{\lvert \xi \rvert = 1}  \xi^\top (X-Y) \xi \leq \sqrt{2} \bigg[ 2m + 2 C + \sup_{\lvert \xi \rvert = 1}  \xi^\top (X+Y) \xi \bigg].  \]
\end{lemma}

\begin{proof}
The inequality \eqref{eq:SorceComp} is unchanged if we multiply each side on the left and on the right by the symmetric matrix $ \begin{pmatrix}
\mathrm{Id} & \mathrm{Id} \\
\mathrm{Id} & -\mathrm{Id}   
\end{pmatrix}$. 
This yields
\[
\begin{pmatrix}
X+Y & X-Y \\
X-Y & X+Y
\end{pmatrix} 
\leq
4C
\begin{pmatrix}
0 & 0 \\
0 & A 
\end{pmatrix}
+
2 m
\begin{pmatrix}
\mathrm{Id} & 0 \\
0 & \mathrm{Id}
\end{pmatrix}.
\]
For any $(t,\xi)$ in $\R \times \R^d$, we apply this to the vector $
\begin{pmatrix}
t \xi & \xi
\end{pmatrix}
$, so that
\[ t^2 \xi^\top (X+Y) \xi + 2t \xi^\top (X-Y) \xi + \xi^\top (X+Y) \xi \leq 4 C \xi^\top A \xi + 2 m ( t^2 + 1) \lvert \xi \rvert^2. \]
Since this holds for every $t$ in $\R$, we get that
\begin{align*}
[ \xi^\top (X-Y) \xi ]^2 &\leq [ \xi^\top (X+Y) \xi - 2m \lvert \xi \rvert^2 ] [ \xi^\top (X+Y) \xi - 4 C \xi^\top A \xi - 2m \lvert \xi \rvert^2] \\
&\leq \tfrac{1}{2} [ \xi^\top (X+Y) \xi - 2m \lvert \xi \rvert^2 ]^2 + \tfrac{1}{2} [ \xi^\top (X+Y) \xi - 4 C \xi^\top A \xi - 2m \lvert \xi \rvert^2]^2.
\end{align*} 
To conclude, we apply this to any normalised vector $\xi$ and we take the square-root of each side.
\end{proof}

\section*{Acknowledgments} 
The author wishes to thank Guy Barles and Cyril Imbert for precious advice on viscosity methods. 
The author also thanks Julien Reygner for careful proofreading of this manuscript. 

\addcontentsline{toc}{section}{References}
\printbibliography

\end{document}